\documentclass[11pt, a4paper]{amsart}
\usepackage{amsmath}%
\usepackage{amsfonts}%
\usepackage{amssymb}%
\usepackage{graphicx}
\usepackage{amsthm}
\usepackage{mathabx}
\usepackage{amssymb}
\usepackage{tikz}
\usepackage{tikz-cd}
\usetikzlibrary{matrix,arrows}
\usepackage[utf8,cp1250]{inputenc}
\usepackage{hyperref}
\usepackage{cite}
\usepackage{footnote}
\usepackage{stmaryrd}
\SetSymbolFont{stmry}{bold}{U}{stmry}{m}{n}
\usepackage[a4paper,margin=3cm]{geometry}
\usepackage{verbatim}
\usepackage{enumerate}
\usepackage[T1]{fontenc}
\newtheorem{thm}{Theorem}
\theoremstyle{plain}

\newtheorem{lem}[thm]{Lemma}

\newtheorem{prop}[thm]{Proposition}

\theoremstyle{remark}

\newtheorem*{acknowledgements}{Acknowledgements}

\newcommand{\BC}{{\mathbb{C}}}

\newcommand{\BH}{{\mathbb{H}}}

\newcommand{\BN}{{\mathbb{N}}}

\newcommand{\BQ}{{\mathbb{Q}}}
\newcommand{\BR}{{\mathbb{R}}}

\newcommand{\CF}{{\mathcal F}}

\renewcommand{\mod}{\mathop{\rm mod}\nolimits}

\newcommand{\im}{\mathop{{\rm Im}}\nolimits}

\newcommand{\GL}[1]{\mathop{\rm GL}_{#1} \nolimits}
\newcommand{\SL}[1]{\mathop{\rm SL}_{#1} \nolimits}

\newcommand{\SO}[1]{\mathop{\rm SO}_{#1} \nolimits}
\newcommand{\Mat}[2]{\mathop{\rm Mat}_{#1 \times #2} \nolimits}

\newcommand{\tr}{\mathop{\rm tr}\nolimits}

\newcommand{\artanh}{\mathop{\rm artanh}\nolimits}

\renewcommand{\Re}{\mathop{\rm Re}\nolimits}
\renewcommand{\Im}{\mathop{\rm Im}\nolimits}

\newcommand{\quotient}[2]{
        \mathchoice
            {% \displaystyle
                \text{\raise1ex\hbox{$#1$}\Big/\lower1ex\hbox{$#2$}}%
            }
            {% \textstyle
                #1\,/\,#2
            }
            {% \scriptstyle
                #1\,/\,#2
            }
            {% \scriptscriptstyle  
                #1\,/\,#2
            }
    }
		
\newcommand{\lquotient}[2]{
        \mathchoice
            {% \displaystyle
                \text{\lower1ex\hbox{$#1$}\Big \backslash \raise01ex\hbox{$#2$}}%
            }
            {% \textstyle
                #1\,\backslash\,#2
            }
            {% \scriptstyle
                #1\,\backslash\,#2
            }
            {% \scriptscriptstyle  
                #1\,\backslash\,#2
            }
    }

\newcommand{\rquotient}[2]{
        \mathchoice
            {% \displaystyle
                \text{\raise01ex\hbox{$#1$}\Big/\lower1ex\hbox{$#2$}}%
            }
						{% \textstyle
                #1\,/\,#2
            }
            {% \scriptstyle
                #1\,/\,#2
            }
            {% \scriptscriptstyle  
                #1\,/\,#2
            }
    }
		
\newcommand{\lrquotient}[3]{
        \mathchoice
            {% \displaystyle
                \text{\lower1ex\hbox{$#1$}\Big \backslash \raise01ex\hbox{$#2$}\Big/\lower1ex\hbox{$#3$}}%
            }
            {% \textstyle
                #1\,\backslash\,#2\,/\,#3
            }
            {% \scriptstyle
                #1\,\backslash\,#2\,/\,#3
            }
            {% \scriptscriptstyle  
                #1\,\backslash\,#2\,/\,#3
            }
    }

\newcommand{\sm}{\left(\begin{smallmatrix}}
\newcommand{\esm}{\end{smallmatrix}\right)}
\newcommand{\bpm}{\begin{pmatrix}}
\newcommand{\ebpm}{\end{pmatrix}}

\newcommand{\one}{{\rm 1\mskip-4mu l}}

\numberwithin{equation}{section}
\begin{document}
%\selectlanguage{english}

\bibliographystyle{plain}

\date{\today}

\title[Small diameters and generators]{Small diameters and generators for arithmetic lattices in $\mathrm{SL}_2(\mathbb{R})$ and certain Ramanujan graphs}

\author{Raphael S. Steiner}
\address{Department of Mathematics, ETH Z{\"u}rich, 8092 Z\"urich, CH}%
\email{raphael.steiner.academic@gmail.com}%

%ArXiv, Number Theory, Geometric Topology, Dynamics, Group Theory

%\thanks{Thanks for Author One.}

%\date{Mai 20, 2014}
%\subjclass[2010]{11L05 (11L07, 11F72)}
%\keywords{}

%\dedicatory{Dedicated to Professor XY on the occasion of his seventieth birthday.}

\begin{abstract} We show that arithmetic lattices in $\mathrm{SL}_{2}(\mathbb{R})$, stemming from the proper units of an Eichler order in an indefinite quaternion algebra over $\mathbb{Q}$, admit a `small' covering set. In particular, we give bounds on the diameter if the quotient space is co-compact. Consequently, we show that these lattices admit small generators. Our techniques also apply to definite quaternion algebras where we show Ramanujan-strength bounds on the diameter of certain Ramanujan graphs without the use of the Ramanujan bound.
\end{abstract}
\maketitle

%\begin{acknowledgement} Thank you
%\end{acknowledgement}

\section{Introduction}

Let $B$ be an indefinite quaternion algebra over $\mathbb{Q}$ of reduced discriminant\footnote{Product of the finite primes at which $B$ ramifies.} $D$ and $R \subset B$ be an Eichler order of level $Q$. Let $\Gamma$ be the subset of proper units (elements of norm one) of $R$. Further, fix an isomorphism $B(\mathbb{R})$ with the matrix algebra $\Mat{2}{2}(\mathbb{R})$\footnote{Any other isomorphism is a conjugate by a matrix in $\GL{2}(\mathbb{R})$.}. In the case where $B$ is already split over $\mathbb{Q}$ ($D=1$), we may choose this identification such that
$$
\Gamma = \Gamma_0(Q) = \{ \sm a & b \\ c & d \esm \in \SL{2}(\mathbb{Z}) \, | \, c \equiv 0 \mod(Q)\}
$$
is the familiar congruence lattice. Under the identification of $B(\mathbb{R})$, $\Gamma$ is a lattice in $\SL{2}(\mathbb{R})$ of co-volume $V_{\Gamma} = (DQ)^{1+o(1)}$. Furthermore, $\Gamma$ is co-compact if and only if $B(\mathbb{Q})$ is a division algebra ($D\neq 1$) \cite{Eichlermodcorr}.

A typical example of a fundamental domain for the action of $\Gamma$ on the homogeneous space $\rquotient{\SL{2}(\mathbb{R})}{\SO{2}(\mathbb{R})} \cong \mathbb{H}$, the upper half-plane, is a normal polygon, also known as a Dirichlet domain, which is given as follows. Let $w \in \mathbb{H}$ be a point whose stabiliser $\Gamma_w$ in $\Gamma$ consists only of plus and minus the identity element. Then, the normal polygon with centre $w$ is given by
\begin{equation}
\mathcal{F}_{\Gamma,w} = \{ z \in \mathbb{H} \, | \, d(z,w) < d(\gamma z, w), \ \forall \gamma \in \Gamma, \gamma \neq \pm I \},
\label{eq:normalpolygon}
\end{equation}
where $d$ denotes the hyperbolic distance. If $B$ is split and $\Gamma = \Gamma_0(Q)$, then another typical example of a fundamental domain is given by the standard polygon
\begin{equation}
\mathcal{F}_{\Gamma,\infty} = \{ z \in \mathbb{H} \, | \, |\Re(z)| < \tfrac{1}{2}, \, \Im(z) > \Im(\gamma z), \ \forall \gamma \in \Gamma-\Gamma_{\infty} \},
\label{eq:standardpolygon}
\end{equation}
also referred to as a Ford domain. The sides of these polygons may be paired up such that the corresponding side-pairing motions together with $-I$ generate the group $\Gamma$ \cite[Chapter 2]{IwSpecMeth}. Thus, it comes as no surprise that the 'size' of these fundamental domains, e.g.\@ the diameter of $\lquotient{\Gamma}{\mathbb{H}}$ if the latter is compact, are related to the size of generators of $\Gamma$. Algorithms to compute fundamental domains and subsequently a set of generators have been devised by Johansson \cite{Johanssonfund}, Voight \cite{Voightfundamental}, and subsequently improved by Rickards \cite{Rickardsfundcomp} if $\Gamma$ is co-compact, and Kurth--Long \cite{KurthLong} if $\Gamma$ is a finite index subgroup of $\SL{2}(\mathbb{Z})$ through the use of Farey symbols \cite{FareySymbol}. In the latter case, further algorithms based on the Reidemeister--Schreier process \cite{Reidemeister,Schreier} are available to determine an independent set of generators for $\Gamma(p)$, where $p$ is a prime, by Frasch \cite{Frasch33}, for $\Gamma_0(p)$ by Rademacher \cite{Rademacher29}, and for $\Gamma_0(Q)$, for general $Q\in \mathbb{N}$, by Chuman \cite{Chuman73}. Albeit the former algorithms due Johansson, Voight, Rickards, and Kurth--Long work well in practise, they don't give any answer to the question regarding the asymptotic size of the (produced) generators and thus an upper bound on their time complexity. The algorithms based on the Reidemeister--Schreier process do give an explicit set of generators whose elements are of polynomial size in the co-volume, but they are far from the generating set whose elements are of least size. Stronger results in that direction were given by Khoai \cite{Khoai}, who managed to show that $\Gamma_0(p^r)$ is generated by elements of Frobenius norm bounded by $O(p^{2r})$, and Chu--Li \cite{ChuLi} who managed to show that $\Gamma$ co-compact is generated by its elements of Frobenius norm bounded by $O_{\epsilon}(V_{\Gamma}^{2.56+\epsilon})$\footnote{They state their theorem with exponent $7.68$, though their method gives $5.12+o(1)$. In turn, this can be halved again by replacing their final argument by the argument in this paper.}. In this paper, we shall prove the following theorems.

\begin{thm} $\Gamma_0(Q)$ is generated by its elements of Frobenius norm $O_{\epsilon}(Q^{1+\epsilon})$.
\label{thm:gennoncompact}
\end{thm}

\begin{thm} Let $\Gamma \subset \SL{2}(\mathbb{R})$ be a co-compact arithmetic lattice of co-volume $V_{\Gamma}$ stemming from the proper units of an Eichler order $R$ of level $Q$ in a quaternion algebra $B$ over $\mathbb{Q}$ of reduced discriminant $D$. Then, for almost every $\sigma \in \lquotient{\Gamma}{\SL{2}(\mathbb{R})}$, $\sigma^{-1}\Gamma\sigma$ is generated by its elements of Frobenius norm $O_{\epsilon}(V_{\Gamma}^{2+\epsilon})$. In other words, almost every embedding $\Gamma$ of the proper units of the order $R$ into $\SL{2}(\mathbb{R})$ is generated by its elements of Frobenius norm $O_{\epsilon}(V_{\Gamma}^{2+\epsilon})$.

If one assumes either of the following conditions:
\begin{itemize}
	\item $Q$ is square-free,
	\item Selberg's eigenvalue conjecture for $\Gamma_0(DQ)$,
	\item the sup-norm conjecture in the level aspect for exceptional eigenforms on \\ $\lrquotient{\Gamma}{\SL{2}(\mathbb{R})}{\SO{2}(\mathbb{R})}$,
\end{itemize}
then, $\Gamma$ is in fact generated by its elements of Frobenius norm $O_{\epsilon}(V_{\Gamma}^{2+\epsilon})$ regardless of the embedding.
\label{thm:gencompact}
\end{thm}
Theorem \ref{thm:gennoncompact} follows from carefully bounding the standard polygon \eqref{eq:standardpolygon} by isometric circles of large radius. This is layed out in Section \ref{sec:stdpolygon}. Theorem \ref{thm:gencompact} follows from showing, that the normal polygon \eqref{eq:normalpolygon} is contained in a ball of small radius.

\begin{thm} Let $\Gamma \subset \SL{2}(\mathbb{R})$ be a co-compact arithmetic lattice of co-volume $V_{\Gamma}$ stemming from the proper units of an Eichler order $R$ of level $Q$ in a quaternion algebra $B$ over $\mathbb{Q}$ of reduced discriminant $D$. Then, for almost every $w \in \lquotient{\Gamma}{\mathbb{H}}$
\begin{equation}
\sup_{z \in \mathbb{H}} \min_{\gamma \in \Gamma} d(\gamma z, w) \le (2+o(1)) \log 3V_{\Gamma}.
\label{eq:thm3}
\end{equation}
If one assumes either of the following conditions:
\begin{itemize}
	\item $Q$ is square-free,
	\item Selberg's eigenvalue conjecture for $\Gamma_0(DQ)$,
	\item the sup-norm conjecture in the level aspect for exceptional eigenforms on \\ $\lrquotient{\Gamma}{\SL{2}(\mathbb{R})}{\SO{2}(\mathbb{R})}$,
\end{itemize}
then, for every $w \in \mathbb{H}$ almost every $z \in \lquotient{\Gamma}{\mathbb{H}}$ satisfies
\begin{equation}
\min_{\gamma \in \Gamma} d(\gamma z, w) \le (1+o(1)) \log 3V_{\Gamma}.
\label{eq:thm3eq2}
\end{equation}
In particular, the diameter of the hyperbolic surface $\lquotient{\Gamma}{\mathbb{H}}$ is bounded by $(2+o(1)) \log 3 V_{\Gamma}$.
\label{thm:diamcompact}
\end{thm}
The bound \eqref{eq:thm3eq2} for the almost diameter is sharp and one may speculate whether the actual diameter is around the same length. The latter would imply that the co-compact lattices under consideration are generated by its elements of norm $O_{\epsilon}(V_{\Gamma}^{1+\epsilon})$ and thus would bring it onto equal footing with Theorem \ref{thm:gennoncompact}. We should further remark that the bound depends on some Siegel-zero estimates and is thus not effective, respectively can be made effective with at most one exception. %We would further like to remark that all constants are in principal effective and that the quantities $V_{\Gamma}^{\epsilon}$ could be made explicit in terms of divisor bounds and powers of $\log(3V_{\Gamma})$.

In order to prove Theorem \ref{thm:gencompact}, one could make use of Ratner's exponential mixing for $\SL{2}(\mathbb{R})$ \cite{RatnerSL2mixing}. This approach was taken by Chu--Li \cite{ChuLi}, who showed \eqref{eq:thm3} for every $w \in \mathbb{H}$ with $2$ replaced by $2.56$. Instead, we simplify the argument a bit by using the operator which averages over a sphere of a given radius. The latter operator was employed by Golubev--Kamber \cite{KamberCutoff} who showed, amongst many results of related nature, that the almost diameter under the Selberg eigenvalue conjecture is bounded by $\log(3V_{\Gamma})+(2+o(1))\log\log(9V_{\Gamma})$ under the mild assumption that $\lquotient{\Gamma}{\mathbb{H}}$ has not too many points of small injectivity radius.

%In order to prove Theorem \ref{thm:gencompact}, one could make use of Ratner's exponential mixing for $\SL{2}(\mathbb{R})$ \cite{RatnerSL2mixing}. This approach was taken by Chu--Li \cite{ChuLi}, who showed \eqref{eq:thm3} for every $w \in \mathbb{H}$ with $2$ replaced by $2.56$. Instead, we simplify the argument a bit by using the operator which averages over a sphere of a given radius.

Our improvement compared to previous results comes from the incorporation of a density estimate for exceptional eigenvalues as well as the newly available fourth moment bound for Maass forms in the level aspect by Khayutin--Nelson--Steiner \cite{Maasstheta4moment}.

%In closing this introduction, we would also like to touch on the close connection of the problem considered here, of bounding the (almost) diameter of hyperbolic surfaces, to its non-archimedean counterpart, bounding the diameter of LPS- or more generally expander graphs.

At last, we would also like to touch on the closely related problem of bounding the diameter of expander graphs. A rich family of expander graphs, so-called Ramanujan graphs, may be constructed from arithmetic data associated to \emph{definite} quaternion algebras \cite{LPSGraph, PizerRamanujan}. These Ramanujan graphs (by definition) enjoy a large spectral gap, which in turn yields a small upper bound on the diameter. The incorporation of a density estimate for large eigenvalues for (homogeneous) expander graphs has proved valuable in showing that they admit a smaller diameter than what could be directly inferred from their spectral gap \cite{KamberDens}. In Section \S\ref{sec:graphs}, we demonstrate the usefulness of the fourth moment bound of Khayutin--Nelson--Steiner \cite{Maasstheta4moment} also in the context of expander graphs, by proving upper bounds on the diameter of certain Ramanujan graphs \emph{without} the use of the Ramanujan bound which are of \emph{equal} strength.

%XXX Sarnak's letter, graphs etc.
%The latter question of bounding the diameter of the hyperbolic surface $\lquotient{\Gamma}{\mathbb{H}}$ is closely related to the cut-off phenomenon on the distribution of distances on $\lquotient{\Gamma}{\mathbb{H}}$ \cite{KamberCutoff} and has further analogues in the form diameters/cut-offs of expander graphs \cite{KamberDens}. In these contexts, inclusion of density estimates for the exceptional spectrum has likewise proved useful.

%Touch graphs, result Kamber, LPS, Sarnak letter? Stronger density + Fourthmoment bound

\begin{acknowledgements} I am grateful to Ilya Khayutin and Paul Nelson for fruitful conversations on this and related topics as well as comments on earlier drafts. I am further thankful to Peter Sarnak and Amitay Kamber for their encouragement and conversations on the subject, as well as Davide Ravotti and Miko\l{}aj Fr\k{a}czyk for clarifying various results that went into this paper.

The work on this manuscript began at the Institute for Advanced Study, where I was supported by the National Science Foundation Grant No. DMS -- 1638352 and the Giorgio and Elena Petronio Fellowship Fund II, and completed at the Institute for Mathematical Research (FIM) at ETH Z\"urich.
\end{acknowledgements}

\section{Co-Compact Lattices}\label{sec:co-comp} Write $G=\SL{2}(\mathbb{R})$ and $K=\SO{2}(\mathbb{R})$ for short. Let $\mu$ denote the Haar measure on $G$ normalised such that $d\mu(n(x)a(y)k(\theta))=\frac{1}{y^2}dxdyd\theta$, where
$$
n(x) = \begin{pmatrix} 1 & x \\ 0 & 1\end{pmatrix}, \quad a(y) = \begin{pmatrix} y^{\frac{1}{2}} & 0 \\ 0 & y^{-\frac{1}{2}} \end{pmatrix}, \quad k(\theta) = \begin{pmatrix} \cos \theta & -\sin \theta   \\ \sin \theta & \cos \theta \end{pmatrix}.
$$
Let $\Gamma$ be a co-compact lattice as in the introduction. Denote by $V_{\Gamma}$ the co-volume of $\Gamma$ with respect to $\mu$. Further, $\mu$ descends to a finite measure $\nu$ on $\lquotient{\Gamma}{G}$, which we normalise to a probability measure. With $\nu_{\star}$ we denote the push forward of $\nu$ to $\lrquotient{\Gamma}{G}{K}$. Let $\{u_j\}_j$ be an orthonormal basis of Hecke--Maass forms on $L^2(\lrquotient{\Gamma}{G}{K},\nu_{\star})$, which we may also regard as an orthonormal basis of the $K$-invariant subspace of $L^2(\lquotient{\Gamma}{G},\nu)$. We denote the Laplace eigenvalue of $u_j$ with $-\lambda_j=-(\frac{1}{4}+t_j^2)$, where $t_j \in \mathbb{R} \cup i[0,\frac{1}{2}]$. We let $d$ denote the hyperbolic distance on the upper half-plane $\mathbb{H} \cong \rquotient{G}{K}$ and $u$ the related quantity
$$
u(z,w)= \tfrac{1}{2}(\cosh (d(z,w))-1) = \frac{|w-z|^2}{4 \im(z) \im(w)}.
$$
We note that 
$$
u(z,w) = \tfrac{1}{4} \tr(gg^{t})-\tfrac{1}{2} =: u(g)
$$
where $g$ is any matrix that takes $z$ to $w$ and $u:G \to \mathbb{R}^+_0$ is left and right $K$-invariant. Let $S: \mathbb{R}^{+}_0 \to [0,1]$ be a smooth bump function supported on $[0,\delta]$ for some small $\delta>0$, such that its Selberg/Harish-Chandra transform
$$
h(t) = 4\pi \int_{0}^{\infty} S(u) \cdot {}_2F_1(\tfrac{1}{2}+it,\tfrac{1}{2}-it;1;-u) du
$$
is non-negative and 
\begin{equation}
h(\pm \tfrac{i}{2}) = 4 \pi \int_0^{\infty} S(u) du \asymp 1,
\label{eq:htrivial}
\end{equation}
where the implied constants may depend on $\delta$. We note that 
\begin{equation}
|h(t)| \le 4\pi \int_0^{\infty} |S(u)| du = h( \pm \tfrac{i}{2} ) \ll 1.
\label{eq:hineq}
\end{equation}
Let 
$$
B(g_1,g_2) = \sum_{\gamma \in \Gamma} S( u(g_2^{-1}\gamma g_1)),
$$
where $S(u(g_2^{-1}g_1))$ is a point-pair invariant, and thus we get the spectral expansion (cf. \cite[Theorem 1.14]{IwSpecMeth})
\begin{equation}
B(g_1,g_2) = \frac{2 \pi}{ V_{\Gamma} } \sum_j h(t_j) u_j(g_1)\overline{u_j(g_2)}.
\label{eq:ballspecexp}
\end{equation}
$B_g:=B(g,\cdot)$ will take the role of a smooth right $K$-invariant ball on $\lquotient{\Gamma}{G}$. The following lemma is of crucial importance. Essentially, it says that $B_g$ may be compared to an Euclidean ball.
\begin{lem} Suppose $\delta$ is sufficiently small but (strictly) positive. Then, we have for any $g_1 \in G$ that
$$
\int_{\lquotient{\Gamma}{G}} |B(g_1,g_2)|^2 d\nu(g_2) \ll \frac{1}{V_{\Gamma}},
$$
where the implied constant may depend on $\delta$.
\label{lem:KernelL2bound}
\end{lem}
\begin{proof} We have
$$\begin{aligned}
\int_{\lquotient{\Gamma}{G}} |B(g_1,g_2)|^2 d\nu(g_2) &= \frac{4\pi^2}{V_{\Gamma}^2} \sum_j |h(t_j)|^2 |u_j(g_1)|^2 \\
&= \frac{2\pi}{V_{\Gamma}} \sum_{\gamma \in \Gamma} ((S \circ u) \star (S \circ u))(g_1^{-1}\gamma g_1), 
\end{aligned}$$
where the convolution is taken on $G$. We note that $(S \circ u) \star (S \circ u)$ is bounded for fixed $\delta$. Moreover, it is supported on $g \in G$ with $u(g) \le 4\delta(1+\delta)$. We now note that $u(g) \ge \frac{(\tr g)^2}{4}-1$. Thus, if $\delta$ is sufficiently small the sum over all hyperbolic $\gamma \in \Gamma$ is zero since for those $\tr \gamma$ is at least $3$. Now again for $\delta$ sufficiently small, by the Margulis' Lemma (cf. \cite[Section 3.1]{MikolajsmallBetti}), the subgroup $\widetilde{\Gamma}$ generated by the remaining $\gamma$ for which $u(g_1^{-1}\gamma g_1) \le 4\delta(1+\delta)$ is virtually abelian and in particular of one of the following types
\begin{enumerate}[(i)]
	\item An infinite cyclic group generated by an hyperbolic or parabolic isometry;
	\item A finite cyclic group generated by an elliptic isometry;
	\item An infinite dihedral group generated by two elliptic isometries of
order 2.
\end{enumerate}
The first case only contains the elliptic elements plus minus the identity. In the second case, $\widetilde{\Gamma}$ is finite and of order bounded by $6$ as the characteristic polynomial is a cyclotomic polynomial of degree at most two over $\mathbb{Q}$. In the third case, the subgroup $\widetilde{\Gamma}$ fixes a geodesic and all elliptic elements not equal to plus minus the identity fix a single point on this geodesic (and perform a rotation by $\pi$ around the point). These fixpoints are evenly spread along the geodesic by the distance corresponding to the translation of the minimal hyperbolic element in $\widetilde{\Gamma}$. Thus, if $\delta$ is small enough there are at most four elliptic elements $\gamma \in \widetilde{\Gamma}$ such $u(g_1^{-1}\gamma g_1) \le 4\delta(1+\delta)$. We conclude the lemma.
\end{proof} %We postpone its proof to Section \ref{sec:L2ball}.

%Let $\alpha_T= \sm e^{\frac{T}{2}} & \\ & e^{-\frac{T}{2}} \esm$ denote the geodesic flow on $\lquotient{\Gamma}{G}$. Our aim is to show the following proposition.
%\begin{prop} We have
%\begin{multline}
%\int_{\lquotient{\Gamma}{G}} \sup_{g_2 \in \lrquotient{\Gamma}{G}{K}} \left| \langle \alpha_T \circ B_{g_1} , B_{g_2} \rangle-\langle B_{g_1} , u_0 \rangle \overline{\langle B_{g_2} , u_0 \rangle } \right|^2 d\nu(g_1) \\
%\ll_{\epsilon} T^2 V_{\Gamma}^{\epsilon} \left(e^{-T} V_{\Gamma}^{-2} + e^{-\frac{T}{2}} V_{\Gamma}^{-3} \right).
%\label{eq:meansquare}
%\end{multline}
%If one assumes either of the following conditions:
%\begin{itemize}
%	\item $Q$ is square-free,
%	\item Selberg's eigenvalue conjecture for $\Gamma_0(DQ)$,
%	\item the sup-norm conjecture in the level aspect for exceptional eigenforms on \\ $\lrquotient{\Gamma}{\SL{2}(\mathbb{R})}{\SO{2}(\mathbb{R})}$,
%\end{itemize}
%then, we have uniformly for all $g_1,g_2 \in G$ that
%\begin{equation}
%\langle \alpha_T \circ B_{g_1} , B_{g_2} \rangle-\langle B_{g_1} , u_0 \rangle \overline{\langle B_{g_2} , u_0 \rangle } \ll_{\epsilon} T V_{\Gamma}^{\epsilon} \left( e^{-\frac{T}{2}} V_{\Gamma}^{-1} + e^{-\frac{T}{4}} V_{\Gamma}^{-\frac{3}{2}} \right).
%\label{eq:absgeoflow}
%\end{equation}
%\label{prop:maincomp}
%\end{prop}

%\begin{proof}[Proof of Proposition \ref{prop:maincomp}] It follows from Ratner's work on exponential mixing \cite{RatnerSL2mixing} (see \cite{EffSL2Mixing} for an effective version) 

We now get to the heart of the argument. We shall use an averaging operator which averages over a sphere of radius $T$: 
$$
A_Tf(z) = \frac{1}{2\pi} \int_0^{2 \pi} f\left(g_z k(\theta) \sm e^{\frac{T}{2}} & 0 \\ 0 & e^{-\frac{T}{2}} \esm  \right) d\theta,
$$
where $g_z$ is any matrix that takes $i$ to $z$.

We shall prove the following proposition.

\begin{prop} For $T\ge 1$, we have
\begin{multline}
\int_{\lquotient{\Gamma}{G}} \sup_{g_2 \in \lrquotient{\Gamma}{G}{K}} \left| \langle A_T B_{g_1} , B_{g_2} \rangle-\langle B_{g_1} , u_0 \rangle \overline{\langle B_{g_2} , u_0 \rangle } \right|^2 d\nu(g_1) \\
\ll_{\epsilon} T^2 V_{\Gamma}^{-2+\epsilon} \left(e^{-T}  + e^{-\frac{T}{2}} V_{\Gamma}^{-1} \right).
\label{eq:meansquare}
\end{multline}
Assume either of the following conditions:
\begin{itemize}
	\item $Q$ is square-free,
	\item Selberg's eigenvalue conjecture for $\Gamma_0(DQ)$,
	\item the sup-norm conjecture in the level aspect for exceptional eigenforms on \\ $\lrquotient{\Gamma}{\SL{2}(\mathbb{R})}{\SO{2}(\mathbb{R})}$,
\end{itemize}
then we have the stronger bound
\begin{equation}
\|A_TB_g-\langle B_g,u_0 \rangle u_0\|_2^2  \ll_{\epsilon} T^2V_{\Gamma}^{-2+\epsilon} \left( e^{-T}V_{\Gamma}+e^{-\frac{T}{2}}V_{\Gamma}^{\frac{1}{2}}  \right).
\label{eq:L2bound}
\end{equation}
\label{prop:ballavg}
\end{prop}

Before proceeding with the proof of the proposition, we shall show how Theorem \ref{thm:diamcompact} follows from it. The first half of the Theorem follows from choosing $T=(2+o(1)) \log 3 V_{\Gamma}$ in \eqref{eq:meansquare}. As the main term
$$
\langle  B_{g_1} , u_0 \rangle \overline{\langle B_{g_2} , u_0 \rangle } = \frac{4 \pi^2}{V_{\Gamma}^2} |h(\tfrac{i}{2})|^2 \asymp V_{\Gamma}^{-2},
$$
it follows that $\langle A_T B_{g_1} , B_{g_2} \rangle>0$ for most $g_1$. In order to prove the second half of the theorem, we shall use \eqref{eq:L2bound} with $T_0=(1+o(1)) \log 3 V_{\Gamma}$. We then find
$$
\nu_{\star}(\{z \in {\textstyle \lrquotient{\Gamma}{G}{K}} \, | \, A_{T_0}B_g(z)=0 \}) \ll V_{\Gamma}^2 \|A_{T_0}B_g-\langle B_g,u_0 \rangle u_0\|_2^2 = o(1).
$$

\begin{proof}[Proof of Prop. \ref{prop:ballavg}] We first note that each $u_i$ is an eigenfunction of $A_T$ with eigenvalue ${}_2F_1(\frac{1}{2}-it_i,\frac{1}{2}+it_i;1;\frac{1}{2}-\frac{1}{2}\cosh(T))=P_{-\frac{1}{2}-it_i}(\cosh(T))$, cf. \cite[Corollary 1.13]{IwSpecMeth}. We have that this eigenvalue is bounded by
$$
\begin{cases} (T+1)e^{-\frac{T}{2}(1-\sqrt{1-4\lambda_i})}, & 0\le \lambda_i \le \frac{1}{4}, \\
(T+1)e^{-\frac{T}{2}}, & \frac{1}{4} \le \lambda_i,  \end{cases}
$$
cf. \cite[Prop. 2.3 \& 7.2]{KamberCutoff}. Hence, after referring to the spectral expansion \eqref{eq:ballspecexp}, we find that
\begin{multline}
\langle A_{T} B_{g_1} , B_{g_2} \rangle - \langle  B_{g_1} , u_0 \rangle \overline{\langle B_{g_2} , u_0 \rangle }   \\
= O \left( \frac{1}{V_{\Gamma}^2} \sum_{0<\lambda_j < \frac{1}{4}} T\left(e^{-\frac{T}{2}} \right)^{1-\sqrt{1-4\lambda_j}} |h(t_j)|^2 |u_j(g_1)u_j(g_2)| \right)+O\left( Te^{-\frac{T}{2}} \|B_{g_1}\|_2 \|B_{g_2}\|_2 \right)
\label{eq:Ratnerexpmixing}
\end{multline}
for any two $g_1,g_2 \in G$ and $T \ge 1$. We note that the second smallest eigenvalue $\lambda_1 \ge \frac{3}{16}$, due to Selberg \cite{SelbergFourier} and the Jacquet--Langlands correspondence \cite{JacquetLanglands}. By Lemma \ref{lem:KernelL2bound}, we find that the second error term is bounded by $O(Te^{-\frac{T}{2}} V_{\Gamma}^{-1})$. It remains to deal with the first error term. 
By Cauchy--Schwarz, we may estimate
\begin{multline*}
\frac{1}{V_{\Gamma}^2} \sum_{0<\lambda_j < \frac{1}{4}} T\left(e^{-\frac{T}{2}} \right)^{1-\sqrt{1-4\lambda_j}} |h(t_j)|^2 |u_j(g_1)u_j(g_2)| \\
\le \left(\frac{1}{V_\Gamma^2} \sum_{\frac{3}{16}\le \lambda_j < \frac{1}{4}} |h(t_j)|^2 |u_j(g_2)|^2 \right)^{\frac{1}{2}}  \left(\frac{1}{V_\Gamma^2} \sum_{\frac{3}{16}\le \lambda_j < \frac{1}{4}} T^2\left(e^{-\frac{T}{2}}\right)^{2(1-\sqrt{1-4\lambda_j})} |h(t_j)|^2 |u_j(g_1)|^2 \right)^{\frac{1}{2}}   \\
\ll \|B_{g_2}\|_2 \left(\frac{1}{V_\Gamma^2} \sum_{\frac{3}{16}\le \lambda_j < \frac{1}{4}} T^2\left(e^{-\frac{T}{2}}\right)^{2(1-\sqrt{1-4\lambda_j})} |h(t_j)|^2 |u_j(g_1)|^2 \right)^{\frac{1}{2}} \\
\ll \frac{1}{V_{\Gamma}^{\frac{1}{2}}} \left(\frac{1}{V_\Gamma^2} \sum_{\frac{3}{16}\le \lambda_j < \frac{1}{4}} T^2\left(e^{-\frac{T}{2}}\right)^{2(1-\sqrt{1-4\lambda_j})} |h(t_j)|^2 |u_j(g_1)|^2 \right)^{\frac{1}{2}}.
%\label{eq:}
\end{multline*}
Thus, we find
\begin{multline*} \int_{\lquotient{\Gamma}{G}} \sup_{g_2 \in \lrquotient{\Gamma}{G}{K}} \left| \langle A_T  B_{g_1} , B_{g_2} \rangle-\langle B_{g_1} , u_0 \rangle \overline{\langle B_{g_2} , u_0 \rangle } \right|^2 d\nu(g_1) \\
\ll \frac{1}{V_{\Gamma}^3} \left( \sum_{\frac{3}{16}\le \lambda_j < \frac{1}{4}} T^2\left(e^{-\frac{T}{2}}\right)^{2(1-\sqrt{1-4\lambda_j})} |h(t_j)|^2 \int_{\lquotient{\Gamma}{G}} |u_j(g_1)|^2 d\nu(g_1) \right) + T^2e^{-T} V_{\Gamma}^{-2} \\
\ll \frac{1}{V_{\Gamma}^3} \left( \sum_{\frac{3}{16}\le \lambda_j < \frac{1}{4}} T^2\left(e^{-\frac{T}{2}}\right)^{2(1-\sqrt{1-4\lambda_j})} \right) + T^2e^{-T} V_{\Gamma}^{-2}.
\end{multline*}
The sum over the exceptional eigenvalues we may bound using the density estimate (cf. \cite[Theorem 11.7]{IwSpecMeth})
\begin{equation}
\sharp \{j>0 \, | \, s_j > \sigma \} \ll_{\epsilon} (DQ)^{3-4\sigma+\epsilon},
\label{eq:density}
\end{equation}
for any $\sigma \ge \frac{1}{2}$, where $s_j=\frac{1}{2}-it_j$. We note that this density estimate, which a priori holds for $\Gamma_0(DQ)$, also holds for the lattice $\Gamma$ under consideration. %This is the case since a newform on $\Gamma$ may be transferred by Jacquet--Langlands to a new form on $\Gamma_0(DQ)$ (with the same eigenvalue). Moreover, this correspondence is one-to-one \cite{} XXXX and the multiplicities of the old forms is bounded by $Q^{o(1)}$ \cite{} XXXXX. 
This is the case since we may consider an explicit Jacquet--Langlands transfer of a form on $\Gamma$ to a form on $\Gamma_0(DQ)$ (which has the same eigenvalue). We note that under any such identification at most $Q^{o(1)}$ forms get mapped onto the same image. This follows from the considerations in \cite{Casselman} and noting that at the places dividing $D$ the corresponding representations are one-dimensional. Using the estimate \eqref{eq:density}, we find
\begin{equation}\begin{aligned}
\sum_{\frac{3}{16}\le \lambda_j < \frac{1}{4}} T^2\left(e^{-\frac{T}{2}}\right)^{2(1-\sqrt{1-4\lambda_j})} &= T^2 \sum_{\frac{1}{2} < s_j \le \frac{3}{4}} \left(e^{\frac{T}{2}}\right)^{4(s_j-1)} \\
&\ll_{\epsilon} (DQ)^{\epsilon} T^2\left(e^{-T} (DQ) + e^{-\frac{T}{2}}\right) \\
& \ll_{\epsilon} V_{\Gamma}^{\epsilon} T^2 \left( e^{-T} V_{\Gamma}+e^{-\frac{T}{2}} \right),
\label{eq:averagedbound}
\end{aligned}\end{equation}
Hence, we conclude the first part of the proposition. For the second part, we need to bound
\begin{multline*}
\|A_TB_g-\langle B_g,u_0 \rangle u_0\|_2^2 \\
 \ll  \frac{1}{V_{\Gamma}^2} \sum_{\frac{3}{16} \le \lambda_j < \frac{1}{4}} T^2\left(e^{-\frac{T}{2}} \right)^{2(1-\sqrt{1-4\lambda_j})} |h(t_j)|^2 |u_j(g)|^2 + T^2e^{-T} \|B_{g}\|_2^2.
\end{multline*}
We note that the second summand is $\ll T^2e^{-T}V_{\Gamma}^{-1}$, which is satisfactory. Furthermore, we find that the first summand is empty if we assume Selberg's eigenvalue conjecture, which takes care of that case. In the other cases, we first recall \eqref{eq:averagedbound}, which implies there is a $g_0 \in \lquotient{\Gamma}{G}$ such that
$$
\frac{1}{V_{\Gamma}^2} \sum_{\frac{3}{16} \le \lambda_j < \frac{1}{4}} T^2\left(e^{-\frac{T}{2}} \right)^{2(1-\sqrt{1-4\lambda_j})} |h(t_j)|^2 |u_j(g_0)|^2 \ll_{\epsilon} T^2 V_{\Gamma}^{-2+\epsilon} \left( e^{-T} V_{\Gamma}+e^{-\frac{T}{2}} \right). 
$$
Thus, in order to prove \eqref{eq:L2bound}, it is sufficient to bound
\begin{multline*}
\frac{1}{V_{\Gamma}^2} \sum_{\frac{3}{16} \le \lambda_j < \frac{1}{4}} T^2\left(e^{-\frac{T}{2}} \right)^{2(1-\sqrt{1-4\lambda_j})} |h(t_j)|^2 \left(|u_j(g)|^2 -|u_j(g_0)|^2\right) \\
\le \frac{1}{V_{\Gamma}^2} \left( \sum_{\frac{3}{16}\le \lambda_j < \frac{1}{4}} T^4\left(e^{-\frac{T}{2}}\right)^{4(1-\sqrt{1-4\lambda_j})} \right)^{\frac{1}{2}} \left( \sum_{\frac{3}{16}\le \lambda_j < \frac{1}{4}} \left(|u_j(g)|^2-|u_j(g_0)|^2 \right)^2 \right)^{\frac{1}{2}}.
\end{multline*}
If we assume that the level $Q$ of $R$ is square-free, then \cite[Theorem 4]{Maasstheta4moment} shows that
$$
\sum_{\frac{3}{16}\le \lambda_j < \frac{1}{4}} \left(|u_j(g)|^2 -|u_j(g_0)|^2 \right)^2 \ll_{\epsilon} V_{\Gamma}^{1+\epsilon}.
$$
The same conclusion also holds if we assume the sup-norm conjecture and referring to Weyl's law. Finally, we may estimate the first factor by once more refering to the density estimate \eqref{eq:density}:
$$
\sum_{\frac{3}{16}\le \lambda_j < \frac{1}{4}} T^4\left(e^{-\frac{T}{2}}\right)^{4(1-\sqrt{1-4\lambda_j})} \ll_{\epsilon} V_{\Gamma}^{\epsilon}T^4 \left( e^{-2T} V_{\Gamma}+e^{-T} \right).
$$
We conclude the proof.
\end{proof}

We are left to infer Theorem \ref{thm:gencompact} from Theorem \ref{thm:diamcompact}. Suppose the stabiliser of the point $i \in \mathbb{H}$ consists only of plus minus the identity and consider the normal polygon $\mathcal{F}_{\Gamma,i}$. Suppose further that $\mathcal{F}_{\Gamma,i}$ is contained in a ball of radius $r$. We shall now translate this picture from the upper half-plane to the Poincar\'e disk through the Cayley transformation $\phi: \mathbb{H} \to \mathfrak{D}$, which maps $z\mapsto (z-i)/(z+i)$. Under this map, $\mathcal{F}_{\Gamma,i}$ gets mapped to a Ford domain $\mathfrak{F}_{\Gamma}$. A motion $\gamma = \sm a & b \\ c & d \esm \in \SL{2}(\mathbb{R})$ gets transferred to the motion 
\begin{equation}
\gamma^{\phi} = \phi \gamma \phi^{-1} = \begin{pmatrix} \frac{a+d}{2}+i\frac{b-c}{2}  & \frac{a-d}{2}-i\frac{b+c}{2} \\ \frac{a-d}{2}+i\frac{b+c}{2} & \frac{a+d}{2}-i\frac{b-c}{2} \end{pmatrix} =: \begin{pmatrix} \overline{F} & \overline{E} \\ E & F\end{pmatrix}
\label{eq:transformed}
\end{equation}
in $\mathrm{SU}(1,1)$. Now, Ford \cite{FordDomain} proved that $\Gamma^{\phi}=\phi \Gamma \phi^{-1}$ is generated by the motions $\gamma^{\phi} \in \Gamma^{\phi}$ whose partial arc of its isometric circle forms part of the boundary of $\mathfrak{F}_{\Gamma}$. We have that the isometric circle corresponding to the motion \eqref{eq:transformed} is given by the equation $|Ez+F|$,
%$$
%\left|(\tfrac{a-d}{2}+i\tfrac{b+c}{2})z+(\tfrac{a+d}{2}-i\tfrac{b-c}{2})\right|=1,
%$$
a circle with radius $1/|E|$ and centre $-F/E$. Thus, in order for the isometric circle of $\gamma^{\phi}$ to intersect $\mathfrak{F}_{\Gamma}$ one must have $(|F|-1)/|E| \le \artanh(\frac{r}{2})$, which after a calculation yields $|E| \le \sinh(r)$ and thus
$$
\|\gamma\|_F^2 = 4|E|^2+2 \ll e^{2r}.
$$
Thus we may conclude Theorem \ref{thm:gencompact} from Theorem \ref{thm:diamcompact} after noting that the points $z\in \mathbb{H}$ with non-trivial stabiliser are a null-set, hence they may be excluded for the first part and for the second part one may conjugate the group by a tiny bit if $i$ happens to be such a point.

\section{Non-Co-Compact Lattices}
\label{sec:stdpolygon}

In this section, we shall bound the standard polygon \eqref{eq:standardpolygon}. The following encapsulation of the standard polygon came about whilst working on \cite{Maasstheta4moment} in order to bound the $L^2$-norm of the theta kernel. In the end, a simpler alternative was found. Thus, I would like thank Ilya Khayutin and Paul Nelson for letting me include the argument in this manuscript and for all of their comments and inputs on previous iterations. We shall also point out that a similar argument has been carried out in the appendix to \cite{GoldenGates}.

We start with two preparatory Lemmata.

%\begin{prop} For any $\epsilon >0$, there is a constant $C_{\epsilon}$ such that the standard polygon $\mathcal{F}_{\Gamma,\infty}$ (cf. Eq. \eqref{eq:standardpolygon}) for $\Gamma=\Gamma_0(Q)$, where $Q$ is a square-free integer, is contained in the union of
%$$
%\{z \in \mathbb{H} \, | \, |\Re(z)|\le \tfrac{1}{2} \text{ and } \Im(z) \ge C_{\epsilon}^{-1} Q^{-1-\epsilon} \}
%$$
%and the union of cuspidal regions around each cusp $\frac{a}{q}$, where $a \in \mathbb{Z}$, $q|Q$ a proper divisor ($q\neq Q$), $(a,q)=1$, and $2|a|\le q $, given ?????
%\label{prop:capsule}
%\end{prop}

%\begin{prop} For any $\epsilon > 0$, we may find a constant $C_{\epsilon}>0$ with the following property. For any $z \in \mathbb{H}$ we may find a $\gamma \in \Gamma_0(N)$ and a cusp $\mathfrak{a}$ of $\Gamma_0(N)$ such that
%$$
%\Im(\sigma_{\mathfrak{a}}^{-1}\gamma z) \ge C_{\epsilon} N^{-\epsilon} \frac{w_{\mathfrak{a}}}{N}.
%$$
%\label{prop:coverset}
%\end{prop}

\begin{lem} For any $\epsilon>0$, we may find a constant $C_{\epsilon}>0$ with the following property. For and two relatively prime integers $a,b$ and natural number $D$, we may find an natural number $k \le C_{\epsilon} D^{\epsilon}$ such that $(a+kb,D)=1$.
\label{lem:Jacobsthal}
\end{lem}
\begin{proof} This is \cite[Lemma 2.1]{saha2019sup}.
\end{proof}

\begin{lem} For any $\epsilon>0$, we may find a constant $C_{\epsilon}>0$ with the following property. For any real number $x$, and natural number $D$, we either have that
\begin{enumerate}[(i)]
	\item there is an integer $c$ such that $|x-c| \le \frac{1}{2(1+C_{\epsilon}D^{\epsilon})}$, or that\\
	\item there is a natural number $b \le 2(1+C_{\epsilon} D^{\epsilon})^2$ and an integer $(a,bD)=1$ such that $$|bx-a| \le \tfrac{1}{2}. $$
\end{enumerate}
\label{lem:fracapprox}
\end{lem}
\begin{proof} By Dirichlet's Approximation Theorem, we may find $d \in \mathbb{N}$ with $d \le K=2(1+C_{\epsilon}D^{\epsilon})$ and $c \in \mathbb{Z}$ such that $|dx-c|\le \frac{1}{K}$. Without loss of generality, we may assume $(c,d)=1$. Now, if $d=1$, then we are done as the first condition is satisfied. Suppose now, that $d \ge 2$, then we may find a pair of integers $a,b$ such that $ad-bc=\pm 1$, where the sign is chosen such that $\frac{a}{b}$ and $x$ lie on the same side of $\frac{c}{d}$ on the number line. We note that the pair of integers $(a+kc,b+kd)$, where $k$ is any integer, also satisfies the same equation. Hence, we may assume $0< b \le d$. Note that we have $(a,c)=1$ and hence we may apply Lemma \ref{lem:Jacobsthal} to even require $(a,bD)=1$ at the cost of increasing the size of $b$ to at most $(1+C_{\epsilon}D^{\epsilon})d$. We have
$$
|x-\tfrac{a}{b}| \le \max\left\{ \frac{1}{dK}, \frac{1}{bd} \right\} \le \frac{1}{2b}. 
$$
The conclusion follows.
\end{proof}

We note that the standard polygon \eqref{eq:standardpolygon} agrees with the Ford domain
\begin{equation}
\mathcal{F}_{\Gamma,\infty} = \{ z \in \mathbb{H}  \,|\,|\Re(z)|\le \tfrac{1}{2} \text{ and } |cQz+d|\ge 1, \, \forall (cQ,d)=1 \}.
\label{eq:stdforddomain}
\end{equation}
We proceed by showing that the isometric circles $|cQz+d|=1$, corresponding to the motion $\sm a & b \\ cQ & d \esm \in \Gamma$, that form part of the boundary of $\mathcal{F}_{\Gamma,\infty}$ must have their radius $1/|cQ|$ bounded below by $Q^{-1+o(1)}$.

Let $z=x+iy \in \mathcal{F}_{\Gamma,\infty}$. We now apply Lemma \ref{lem:fracapprox} to $xQ$ with $D=Q$. Thus, we have either
\begin{enumerate}[(i)]
	\item $|Qx-c| \le \frac{1}{2(1+C_{\epsilon}Q^{\epsilon})}$ for some integer $c$, or
	\item $|bQx-a| \le \frac{1}{2}$ some natural number $b \le 2(1+C_{\epsilon}Q^{\epsilon})^2$ and integer $a$ with $(a,bQ)=1$.
\end{enumerate}
Let us first deal with the second case, which corresponds to $z$ being away from the cusps other than $\infty$. Since, $z$ is in the Ford domain \eqref{eq:stdforddomain} we must have $|bQz-a|\ge 1$ and thus $y \ge \frac{1}{2} Q^{-1} (1+C_{\epsilon}Q^{\epsilon})^{-2}$. Returning to the first case, if $(c,Q)=1$ then once again we must have $|Qz-c|\ge 1$ and hence $y \ge \frac{1}{2} Q^{-1}$. Finally, if $(c,Q)>1$ then we may find natural numbers $k_{\pm} \le C_{\epsilon}Q^{\epsilon}$ such that $(k_{\pm}c \pm 1,Q)=1$. Thus, we have
$$
\frac{k_{-}c-1}{k_{-}Q} < x < \frac{k_{+}c+1}{k_{+}Q}
$$
and the isometric circles $|(k_{\pm}Q)w+(k_{\pm}c\pm 1)|=1$ include the cusp $\frac{c}{Q}$. Hence, we find that every point $z \in \mathcal{F}_{\Gamma,\infty}$ has either $y \ge \frac{1}{2}(1+C_{\epsilon}Q^{\epsilon})^{-2}Q^{-1}$ or is in a cuspidal region in between to isometric circles of radius at least $C_{\epsilon}Q^{-1-\epsilon}$. In particular, we find that every isometric circle, which is part of the boundary of $\mathcal{F}_{\Gamma,\infty}$ must have radius at least $\frac{1}{2}(1+C_{\epsilon}Q^{\epsilon})^{-2}Q^{-1}$. Thus, the side-pairing motions $\sm a & b \\ cQ & d \esm$ of $\mathcal{F}_{\Gamma,\infty}$, which generate $\Gamma$ must have $|cQ| \ll_{\epsilon} Q^{1+2 \epsilon}$, $|a|,|d| \ll_{\epsilon} Q^{1+2\epsilon}$ and consequently $|b| \ll_{\epsilon} Q^{1+4 \epsilon}$ as $ad-bcQ=1$. We conclude Theorem \ref{thm:gennoncompact}.

\section{Graphs}
\label{sec:graphs}

We consider the Brandt--Ihara--Pizer ``super singular isogeny graphs'', $G(p, \ell)$, where $p, \ell$ are primes with $p \equiv 1 \mod(12)$. They are constructed by interpreting Brandt matrices $B(\ell)$ associated to a maximal order $R$ in the quaternion algebra $B_{p,\infty}$ over $\mathbb{Q}$ ramified at exactly $p,\infty$ as adjacency matrices. They constitute a rich family of non-bipartite $(\ell+1)$-regular Ramanujan graphs on $n:=\frac{p-1}{12}+1$ vertices \cite{PizerRamanujan}. Let $f_j \in L^2(G(p,\ell))$, equipped with probability measure, be an orthonormal eigenbasis of the adjacency matrix $B(\ell)$ with $f_0 \equiv 1$. We may and shall also assume that they are eigenfunctions of all other Brandt matrices $B(m)$ for $(m,p)=1$. We shall denote the eigenvalue of $f_j$ with respect to $B(m)$ by $\lambda_j(m)$. By identifying the vertices of $G(p,\ell)$ with the class set $B_{p,\infty}^{\times}\backslash (B_{p,\infty} \otimes \mathbb{A}_f)^{\times} \slash (R \otimes \widehat{\mathbb{Z}})^{\times}$ of $R$, we may interpret the eigenfunctions $f_j$ as automorphic forms in $L^2(\operatorname{PB}_{p,\infty}^{\times}(\mathbb{Q}) \backslash \operatorname{PB}_{p,\infty}^{\times}(\mathbb{A}) \slash K_{\infty}K_f)$, where $K_{\infty}$ is a maximal torus and $K_f$ the projective image of $(R \otimes \widehat{\mathbb{Z}})^{\times}$, constant on each connected component (as a real manifold). The automorphic forms $f_j$ are in one-to-one correspondence with their theta lift, a modular form of weight $2$, level $p$, and trivial character, which is cuspidal if and only if $j \neq 0$. They form a basis of Hecke eigenforms of said space. The $m$-th Hecke eigenvalue of the theta lift of $f_j$ is given by the eigenvalue $\lambda_j(m)$ for $(m,p)=1$ \cite{Basis-Problem-Eich, Basis-Problem-gen}. Thus, by the Petersson trace formula, one has the density estimate (cf.\@ \cite[Eq.\@ (4)]{SarnakNina})
\begin{equation}
	\sum_{j \neq 0} |\lambda_j(m)|^2 \ll_{\epsilon} (p \ell)^{\epsilon} \left(p+m^{\frac{1}{2}}\right).
	\label{eq:graphdensity}
\end{equation}
Likewise, the fourth moment bound \cite[Theorem 5]{Maasstheta4moment} reads
\begin{equation}
	\sup_{x,y \in G(p,\ell)}\sum_{j \neq 0}\left( f_j(x)^2-f_j(y)^2 \right)^2 \ll_{\epsilon} p^{1+\epsilon}.
	\label{eq:def4moment}
\end{equation}

It is known, that the Ramanujan graphs $G(p,\ell)$ have diameter bounded by $(2+o(1))\log_{\ell}(n)$.\footnote{In fact, sharper results are known, see for example \cite{SarnakNestoridi}.} Here, we shall give an alternative proof which avoids using the Ramanujan bound. Instead, we shall make use of the two inequalities \eqref{eq:graphdensity} and \eqref{eq:def4moment}. For $x,y \in G(p,\ell)$, let $K_t(x,y)$ denote the number of non-backtracking random walks of length $t$ from $x$ to $y$. We have the equality (see \cite{HeckeOpS2II})
\begin{equation}
	\sum_{0 \le i \le \frac{t}{2}} K_{t-2i}(x,y) = \frac{1}{n} \sum_j \lambda_j(\ell^t) f_j(x)f_j(y).
	\label{eq:graph}
\end{equation}
If $x,y$ are of distance larger than $t$, then the left-hand side of \eqref{eq:graph} is zero. We find that
\begin{equation}
	\lambda_0(\ell^t) f_0(x)f_0(y) \ll \sum_{j \neq 0} |\lambda_j(\ell^t)| |f_j(x)f_j(y)|,
	\label{eq:graphineq}
\end{equation}
from which we infer
\begin{equation}
	\ell^t \ll \sup_{x \in G(p,\ell)} \sum_{j \neq 0} |\lambda_j(\ell^t)|f_j(x)^2,
	 \label{eq:graphmain}
\end{equation}
since $\lambda_0(\ell^t) = \frac{\ell^{t+1}-1}{\ell -1}$. By orthonormality, we have $\sum_{x \in G(p,\ell)} f_j(x)^2 = n$. Hence, we may bound the right-hand side further
\begin{equation}
	\sup_{x \in G(p,\ell)} \sum_{j \neq 0} |\lambda_j(\ell^t)|f_j(x)^2 \le \sum_{j \neq 0} |\lambda_j(\ell^t)|+ \sup_{x,y \in G(p,\ell)} \sum_{j \neq 0} |\lambda_j(\ell^t)|\left( f_j(x)^2-f_j(y)^2 \right).
	\label{eq:graphdiff}
\end{equation}
By applying Cauchy--Schwarz and making use of \eqref{eq:graph} and \eqref{eq:def4moment}, we conclude $l^{t} \ll_{\epsilon} n^{2+\epsilon}$ or $t \le (2+o(1))\log_{\ell}(n)$. In particular, the diameter is bounded by the same quantity.

\bibliography{RafBib}

\begin{thebibliography}{10}

\bibitem{Casselman}
W.~Casselman.
\newblock On some results of {A}tkin and {L}ehner.
\newblock {\em Math. Ann.}, 201:301--314, 1973.

\bibitem{ChuLi}
M.~Chu and H.~Li.
\newblock Small generators of cocompact arithmetic {F}uchsian groups.
\newblock {\em Proc. Amer. Math. Soc.}, 144(12):5121--5127, 2016.

\bibitem{Chuman73}
Y.~Chuman.
\newblock Generators and relations of {$\Gamma _{0}(N)$}.
\newblock {\em J. Math. Kyoto Univ.}, 13:381--390, 1973.

\bibitem{Eichlermodcorr}
M.~Eichler.
\newblock {\em Lectures on modular correspondences}, volume~56.
\newblock Tata Institute of Fundamental Research Bombay, 1955.

\bibitem{Basis-Problem-Eich}
M.~Eichler.
\newblock The basis problem for modular forms and the traces of the {H}ecke
  operators.
\newblock In {\em Modular functions of one variable, {I} ({P}roc. {I}nternat.
  {S}ummer {S}chool, {U}niv. {A}ntwerp, {A}ntwerp, 1972)}, pages 75--151.
  Lecture Notes in Math., Vol. 320, 1973.

\bibitem{FordDomain}
L.~R. Ford.
\newblock The fundamental region for a {F}uchsian group.
\newblock {\em Bull. Amer. Math. Soc.}, 31(9-10):531--539, 1925.

\bibitem{MikolajsmallBetti}
M.~Fr{\k{a}}czyk and J.~Raimbault.
\newblock Betti numbers of {S}himura curves and arithmetic three-orbifolds.
\newblock {\em Algebra Number Theory}, 13(10):2359--2382, 2019.

\bibitem{Frasch33}
H.~Frasch.
\newblock Die {E}rzeugenden der {H}auptkongruenzgruppen f\"{u}r
  {P}rimzahlstufen.
\newblock {\em Math. Ann.}, 108(1):229--252, 1933.

\bibitem{KamberDens}
K.~Golubev and A.~Kamber.
\newblock {C}utoff on {G}raphs and the {S}arnak--{X}ue {D}ensity of
  {E}igenvalues.
\newblock {\em {P}reprint}, 2019.
\newblock {\tt arXiv:1905.11165}.

\bibitem{KamberCutoff}
K.~Golubev and A.~Kamber.
\newblock Cutoff on hyperbolic surfaces.
\newblock {\em Geom. Dedicata}, 203:225--255, 2019.

\bibitem{Basis-Problem-gen}
Hiroaki Hijikata, Arnold~K. Pizer, and Thomas~R. Shemanske.
\newblock The basis problem for modular forms on {$\Gamma_0(N)$}.
\newblock {\em Mem. Amer. Math. Soc.}, 82(418):vi+159, 1989.

\bibitem{IwSpecMeth}
H.~Iwaniec.
\newblock {\em Spectral methods of automorphic forms}, volume~53 of {\em
  Graduate Studies in Mathematics}.
\newblock American Mathematical Society, Providence, RI; Revista Matem\'atica
  Iberoamericana, Madrid, second edition, 2002.

\bibitem{JacquetLanglands}
H.~Jacquet and R.~P. Langlands.
\newblock {\em Automorphic forms on {${\rm GL}(2)$}}.
\newblock Lecture Notes in Mathematics, Vol. 114. Springer-Verlag, Berlin-New
  York, 1970.

\bibitem{Johanssonfund}
S.~Johansson.
\newblock On fundamental domains of arithmetic {F}uchsian groups.
\newblock {\em Math. Comp.}, 69(229):339--349, 2000.

\bibitem{Maasstheta4moment}
I.~Khayutin, P.~D. Nelson, and R.~S. Steiner.
\newblock {T}heta functions, fourth moments of eigenforms, and the sup-norm
  problem {II}.
\newblock {\em {P}reprint}, 2022.
\newblock {\tt arXiv:2207.12351}.

\bibitem{Khoai}
H.~H. Khoai.
\newblock Sur les s\'eries {$L$} associ\'ees aux formes modularies.
\newblock {\em Bull. Soc. Math. France}, 120(1):1--13, 1992.

\bibitem{FareySymbol}
R.~S. Kulkarni.
\newblock An arithmetic-geometric method in the study of the subgroups of the
  modular group.
\newblock {\em Amer. J. Math.}, 113(6):1053--1133, 1991.

\bibitem{KurthLong}
C.~A. Kurth and L.~Long.
\newblock Computations with finite index subgroups of {${\rm PSL}_2(\Bbb Z)$}
  using {F}arey symbols.
\newblock In {\em Advances in algebra and combinatorics}, pages 225--242. World
  Sci. Publ., Hackensack, NJ, 2008.

\bibitem{HeckeOpS2II}
A.~Lubotzky, R.~Phillips, and P.~Sarnak.
\newblock Hecke operators and distributing points on {$S^2$}. {II}.
\newblock {\em Comm. Pure Appl. Math.}, 40(4):401--420, 1987.

\bibitem{LPSGraph}
A.~Lubotzky, R.~Phillips, and P.~Sarnak.
\newblock Ramanujan graphs.
\newblock {\em Combinatorica}, 8(3):261--277, 1988.

\bibitem{SarnakNestoridi}
E.~Nestoridi and P.~Sarnak.
\newblock Bounded cutoff window for the non-backtracking random walk on
  {R}amanujan graphs.
\newblock {\em {P}reprint}, 2021.
\newblock {\tt arXiv:2103.15176}.

\bibitem{PizerRamanujan}
A.~K. Pizer.
\newblock Ramanujan graphs.
\newblock In {\em Computational perspectives on number theory ({C}hicago, {IL},
  1995)}, volume~7 of {\em AMS/IP Stud. Adv. Math.}, pages 159--178. Amer.
  Math. Soc., Providence, RI, 1998.

\bibitem{Rademacher29}
H.~Rademacher.
\newblock \"{U}ber die {E}rzeugenden von {K}ongruenzuntergruppen der
  {M}odulgruppe.
\newblock {\em Abh. Math. Sem. Univ. Hamburg}, 7(1):134--148, 1929.

\bibitem{RatnerSL2mixing}
M.~Ratner.
\newblock The rate of mixing for geodesic and horocycle flows.
\newblock {\em Ergodic Theory Dynam. Systems}, 7(2):267--288, 1987.

\bibitem{Reidemeister}
K.~Reidemeister.
\newblock Knoten und {G}ruppen.
\newblock {\em Abh. Math. Sem. Univ. Hamburg}, 5(1):7--23, 1927.

\bibitem{Rickardsfundcomp}
J.~Rickards.
\newblock Improved computation of fundamental domains for arithmetic {F}uchsian
  groups.
\newblock {\em {P}reprint}, 2021.
\newblock {\tt arXiv:2110.11503}.

\bibitem{saha2019sup}
A.~Saha.
\newblock Sup-norms of eigenfunctions in the level aspect for compact
  arithmetic surfaces.
\newblock {\em Mathematische Annalen}, pages 1--36, 2019.

\bibitem{GoldenGates}
P.~Sarnak.
\newblock Letter to {S}cott {A}aronson and {A}ndy {P}ollington on the
  {S}olovay--{K}itaev theorem, 2015.

\bibitem{SarnakNina}
P.~Sarnak and N.~Zubrilina.
\newblock Convergence to the {P}lancherel measure of {H}ecke {E}igenvalues.
\newblock {\em {P}reprint}, 2022.
\newblock {\tt arXiv:2201.03523}.

\bibitem{Schreier}
O.~Schreier.
\newblock Die {U}ntergruppen der freien {G}ruppen.
\newblock {\em Abh. Math. Sem. Univ. Hamburg}, 5(1):161--183, 1927.

\bibitem{SelbergFourier}
A.~Selberg.
\newblock On the estimation of {F}ourier coefficients of modular forms.
\newblock In {\em Proc. {S}ympos. {P}ure {M}ath., {V}ol. {VIII}}, pages 1--15.
  Amer. Math. Soc., Providence, R.I., 1965.

\bibitem{Voightfundamental}
J.~Voight.
\newblock Computing fundamental domains for {F}uchsian groups.
\newblock {\em J. Th\'{e}or. Nombres Bordeaux}, 21(2):469--491, 2009.

\end{thebibliography}
\end{document}